\documentclass[12pt]{amsart}
\usepackage{amsmath,amsfonts,amssymb,amsthm}
\usepackage{graphics,color,epsfig}

\newcommand{\Q}{\mathbb Q}
\newcommand{\R}{\mathbb R}
\newcommand{\C}{\mathbb C}
\newcommand{\Z}{\mathbb Z}

\newcommand{\cD}{\mathcal D}
\newcommand{\cF}{\mathcal F}

\newcommand{\cN}{\mathcal N}

\newcommand{\eps}{\varepsilon}
\newcommand{\vhi}{\varphi}
\renewcommand{\Re}{\textnormal{Re~}}
\renewcommand{\Im}{\textnormal{Im~}}

\DeclareMathOperator{\area}{area}

\theoremstyle{plain}
\newtheorem{theorem}{Theorem}[section]

\newtheorem{lemma}[theorem]{Lemma}
\newtheorem{proposition}[theorem]{Proposition}

\theoremstyle{definition}
\newtheorem{definition}[theorem]{Definition}

\theoremstyle{remark}
\newtheorem{remark}[theorem]{Remark}

\begin{document}
\title{Slow Divergence and Unique Ergodicity}
\author{Yitwah Cheung and Alex Eskin}

\address{San Francisco State University \\
San Francisco, California}
\email{cheung@math.sfsu.edu}

\address{University of Chicago \\
Chicago, Illinois}
\email{eskin@math.uchicago.edu}
\date{\today}
\begin{abstract}
In \cite{Ma1} Masur showed that a Teichm\"uller geodesic that 
is recurrent in the moduli space of closed Riemann surfaces is 
necessarily determined by a quadratic differential with a uniquely 
ergodic vertical foliation.  In this paper, we show that a divergent 
Teichm\"uller geodesic satisfying a certain slow rate of divergence 
is also necessarily determined by a quadratic differential with 
unique ergodic vertical foliation.  As an application, we sketch 
a proof of a complete characterization of the set of nonergodic 
directions in any double cover of the flat torus branched over 
two points.  
\end{abstract}
\maketitle

\section{Introduction}
Let $(X,q)$ be a holomorphic quadratic differential.  
The line element $|q|^{1/2}$ induces a flat metric on $X$ which 
  has cone-type singularites at the zeroes of $q$ where the cone 
  angle is a integral multiple of $2\pi$.  
A \emph{saddle connection} in $X$ is a geodesic segment with 
  respect to the flat metric that joins a pair of zeroes of 
  $q$ without passing through one in its interior.  
Our main result is a new criterion for the unique ergodicity 
  of the vertical foliation $\cF_v$, defined by $\Re q^{1/2}=0$.  

\textbf{Teichm\"uller geodesics.}
The complex structure of $X$ is uniquely determined by the atlas 
  $\{(U_\alpha,\vhi_\alpha)\}$ of natural parameters away from 
  the zeroes of $q$ specified by $d\vhi_\alpha=q^{1/2}$.  
The evolution of $X$ under the Teichm\"uller flow is the family 
  of Riemann surfaces $X_t$ obtained by post-composing the charts 
  with the $\R$-linear map $z\to e^{t/2}\Re{z}+ie^{-t/2}\Im{z}$.  
It defines a unit-speed geodesic with respect to the Teichm\"uller 
  metric on the moduli space of compact Riemann surfaces normalised 
  so that Teichm\"uller disks have constant curvature $-1$.  
The Teichm\"uller map $f_t:X\to X_t$ takes rectangles to rectangles 
  of the same area, stretching in the horizontal direction and 
  contracting in the vertical direction by a factor of $e^{t/2}$.  
By a \emph{rectangle} in $X$ we mean a holomorphic map a product 
  of intervals in $\C$ such that $q^{1/2}$ pulls back to $\pm dz$.  
All rectangles are assumed to have horizontal and vertical edges.  

Let $\ell(X_t)$ denote the length of the shortest saddle connection.  
Let $d(t)=-2\log\ell(X_t)$.  Our main result is the following.  

\begin{theorem}\label{thm:main}
There is an $\eps>0$ such that if $d(t)<\eps\log t+C$ for some 
  $C>0$ and for all $t>0$, then $\cF_v$ is uniquely ergodic.  
\end{theorem}

Theorem~\ref{thm:main} was announced in \cite{CE}.  
In \S\ref{S:Networks} we present the main ideas that go into 
  the proof of Theorem~\ref{thm:main} and use them to prove 
  an extension (see Theorem~\ref{thm:recurrent} of Masur's 
  result in \cite{Ma1} asserting that a Teichm\"uller geodesic 
  which accumulates in the moduli space of closed Riemann 
  surfaces is necessarily determined by a uniquely ergodic 
  foliation.  
After briefly discussing the relationship between the various 
  ways of describing rates in \S\ref{S:Rates}, we prove 
  Theorem~\ref{thm:main} in \S\ref{S:Slow}.  
Then, in \S\ref{S:Veech} we sketch the characterisation of the 
  set of nonergodic directions in the double cover of a torus, 
  branched over two points, answering a question of W.~Veech, 
  (\cite{Ve1}, p.32, question 2).

\section{Networks}\label{S:Networks}
If $\nu$ is a (normalised) ergodic invariant measure transverse to 
  the vertical foliation $\cF_v$ then for any horizontal arc $I$ 
  there is a full $\nu$-measure set of points $x\in X$ satisfying 
  \begin{equation}\label{eq:erg:ave}
    \lim \frac{^\#I\cap L_x}{|L_x|} = \nu(I) 
                            \quad\text{as}\quad |L_x|\to\infty
  \end{equation}
  where $L_x$ represents a vertical segment having $x$ as an endpoint.  
Given $I$, the set $E(I)$ of points satisfying (\ref{eq:erg:ave}) for 
  \emph{some} ergodic invariant $\nu$ has full Lebesgue measure.  
We refer to the elements of $E(I)$ as \emph{generic points} and 
  the limit in (\ref{eq:erg:ave}) as the \emph{ergodic average} 
  determined by $x$.  

To prove unique ergodicity we shall show that the ergodic averages 
  determined by all generic points converge to the same limit.  
The ideas in this section were motivated by the proof of Theorem~1.1 
  in \cite{Ma1}.  

\emph{Convention.} When passing to a subsequence $t_n\to\infty$ 
  along the Teichm\"uller geodesic $X_t$ we shall suppress the 
  double subscript notation and write $X_n$ instead of $X_{t_n}$.  
Similarly, we write $f_n$ instead of $f_{t_n}$.  

\begin{lemma}\label{lem:rectangle}
Let $x,y\in E(I)$ and suppose there is a sequence $t_n\to\infty$ 
  such that for every $n$ the images of $x$ and $y$ under $f_n$ 
  lie in a rectangle $R_n\subset X_n$ and the sequence of heights 
  $h_n$ satisfy $\lim h_ne^{t_n/2}=\infty$.  
Then $x$ and $y$ determine the same ergodic averages.  
\end{lemma}
\begin{proof}
One can reduce to the case where $f_n(x)$ and $f_n(y)$ lie at the 
  corners of $R_n$.  Let $n_-$ (resp. $n_+$) be the number of times 
  the left (resp. right) edge of $f_n^{-1}R_n$ intersects $I$.  
Observe that $n_-$ and $n_+$ differ by at most one so that since 
  $h_ne^{t_n/2}\to\infty$, the ergodic averages for $x$ and $y$ 
  approach the same limit.  
\end{proof}

Ergodic averages taken as $T\to\infty$ are determined by fixed 
  fraction of the tail: for any given $\lambda\in(0,1)$ 
  $$\frac{1}{T}\int_0^T f\to c \quad\text{ implies }\quad 
    \frac{1}{(1-\lambda)T}\int_{\lambda T}^T f\to c.$$  
This elementary observation is the motivation behind the following.  
\begin{definition}\label{def:visible}
A point $x$ is $K$-\emph{visible} from a rectangle $R$ if 
  the vertical distance from $x$ to $R$ is at most $K$ 
  times the height of $R$.  
\end{definition}

We have the following generalisation of Lemma~\ref{lem:rectangle}.  
\begin{lemma}\label{lem:visible}
If $x,y\in E(I)$, $t_n\to\infty$ and $K>0$ are such that for every $n$ 
  the images of $x$ and $y$ under $f_n$ are $K$-visible from some 
  rectangle whose height $h_n$ satisfies $h_ne^{t_n/2}\to\infty$, 
  then $x$ and $y$ determine the same ergodic averages.    
\end{lemma}

\begin{definition}\label{def:reachable}
We say two points are $K$-\emph{reachable} from each other if 
  there is a rectangle $R$ from which both are $K$-visible.  
We also say two sets are $K$-\emph{reachable} from each other if 
  every point of one is $K$-reachable from every point of the other.  
\end{definition}

\begin{definition}\label{def:network}
Given a collection $\cN$ of subsets of $X$, we define an undirected 
  graph $\Gamma_K(\cN)$ whose vertex set is $\cN$ and whose edge 
  relation is given by $K$-reachability.  
A subset $Y\subset X$ is said to be $K$-\emph{fully covered} by $\cN$ 
  if every $y\in Y$ is $K$-reachable from some element of $\cN$.  
We say $\cN$ is a $K$-\emph{network} if $\Gamma_K(\cN)$ is connected 
  and $X$ is $K$-fully covered by $\cN$.  
\end{definition}

\begin{proposition}\label{prop:Masur}
If $K>0$, $N>0$, $\delta>0$ and $t_n\to\infty$ are such that 
  for all $n$, there exists a $K$-network $\cN_n$ in $X_n$ 
  consisting of at most $N$ squares, each having measure at 
  least $\delta$, then $\cF_v$ is uniquely ergodic.  
\end{proposition}
\begin{proof}
Suppose $\cF_v$ is not uniquely ergodic.  Then we can find a 
  distinct pair of ergodic invariant measures $\nu_0$ and $\nu_1$ 
  and a horizontal arc $I$ such that $\nu_0(I)\neq\nu_1(I)$.  

We construct a finite set of generic points as follows.  
By allowing repetition, we may assume each $\cN_n$ contains exactly 
  $N$ squares, which shall be enumerated by $A(n,i), i=1,\dots,N$.  
Let $A_1\subset X$ be the set of points whose image under $f_n$ 
  belongs to $A(n,1)$ for infinitely many $n$.  
Note that $A_1$ has measure at least $\delta$ because it is a 
  descending intersection of sets of measure at least $\delta$.  
Hence, $A_1$ contains a generic point; call it $x_1$.  
By passing to a subsequence we can assume the image of $x_1$ 
  lies in $A(n,1)$ for all $n$.  
By a similar process we can find a generic point $x_2$ whose image 
  belongs to $A(n,2)$ for all $n$.  When passing to the subsequence, 
  the generic point $x_1$ retains the property that its image lies 
  in $A(n,1)$ for all $n$.  
Continuing in this manner, we obtain a finite set $F$ consisting of 
  $N$ generic points $x_i$ with the property that the image of $x_i$ 
  under $f_n$ belongs to $A(n,i)$ for all $n$ and $i$.  

Given a nonempty proper subset $F'\subset F$ we can always find 
  a pair of points $x\in F'$ and $y\in F\setminus F'$ such that 
  $f_n(x)$ and $f_n(y)$ are $K$-reachable from each other for 
  infinitely many $n$.  
This follows from the fact that $\Gamma_K(\cN_n)$ is connected.  
By Lemma~\ref{lem:visible}, the points $x$ and $y$ determine 
  the same ergodic averages for any horizontal arc $I$.  
Since $F$ is finite, the same holds for any pair of points in $F$.  

Now let $z_j$ be a generic point whose ergodic average is $\nu_j(I)$, 
  for $j=0,1$.  Since $X_n$ is $K$-fully covered by $\cN_n$, $z_j$ 
  will have the same ergodic average as some point in $F$, which 
  contradicts $\nu_0(I)\neq\nu_1(I)$.  
Therefore, $\cF_v$ must be uniquely ergodic.  
\end{proof}

\begin{definition}\label{def:sep:sys}
Let $(X,q)$ be a holomorphic quadratic differential on a 
  closed Riemann surface of genus at least $2$.  
Two saddle connections are said to be \emph{disjoint} if the 
only points they have in common, if any, are their endpoints.  
We call a collection of pairwise disjoint saddle connections 
a \emph{separating system} if the complement of their union 
has at least two homotopically nontrivial components.  By the 
\emph{length of a separating system} we mean the total length 
of all its saddle connections.  We shall blur the distinction 
between a separating system and the closed subset formed by 
the union of its elements.  
\end{definition}

\begin{definition}\label{def:lengths}
Let $X$ be a closed Riemann surface and $q$ a holomorphic 
  quadratic differential on $X$.  
Let $$\ell_1(X,q)$$ denote the length of the shortest saddle 
  connection in $X$.  
Let $$\ell_2(X,q)$$ denote the infimum of the $q$-lengths of 
  simple closed curves in $X$ that do not bound a disk.  
Let $$\ell_3(X,q)$$ denote the length of the shortest separating 
  system.  
\end{definition}
Observe that $$\ell_1(X,q)\le\ell_2(X,q)\le\ell_3(X,q).$$  

\begin{remark}\label{rem:poles}
Our arguments can also be applied to the case where $X$ is 
  a punctured Riemann surface and $q$ has a simple pole at 
  each puncture.  
A saddle connection is a geodesic segment that joins two 
  singularities (zero or puncture, and not necessarily 
  distinct) without passing through one in its interior.  
In the definition of $\ell_2(X,q)$ the infimum should be 
  taken over simple closed curves that neither bound a 
  disk nor is homotopic to a puncture.  
\end{remark}

\begin{proposition}\label{prop:network}
Let $S$ be a stratum of holomorphic quadratic differentials.  
There are positive constants $K=K(S)$ and $N=N(S)$ such that 
  for any $\delta>0$ there exists $\eps>0$ such that for any 
  area one surface $(X,q)\in S$ satisfying 
\begin{enumerate}
  \item $\ell_3(X,q)>2\delta$, and 
  \item $(X,q)$ admits a complete Delaunay triangulation $\cD$ 
    with the property that the length of every edge is either 
    less than $\eps$ or at least $\delta$ 
\end{enumerate}
  there exists a $K$-network of $N$ embedded squares in $(X,q)$ 
  such that each square has side $\delta$.  
\end{proposition}
\begin{proof}
By a \emph{short} (resp. \emph{long}) edge we mean an edge in 
  $\cD$ of length less than $\eps$ (resp. at least $\delta$.)  
By a \emph{small} (resp. \emph{large}) triangle, we mean a 
  triangle in $\cD$ whose edges are all short (resp. long.)  
Assuming $2\eps<\delta$, we note that all remaining triangles 
  in $\cD$ have one short and two long edges; we refer to them 
  as \emph{medium} triangles.  

To each triangle $\Delta\in\cD$ that has a long edge, i.e. any 
  medium or large triangle, we associate an embedded square of 
  side $\delta$ as follows.  
Note that the circumscribing disk $D$ has diameter at least $\delta$.  
Let $S$ be the largest square concentric with $D$ whose interior 
  is embedded and let $d$ be the length of its diagonal.  
If the boundary of $S$ contains a singularity then $d\ge\delta$.  
Otherwise, there are two segments on the boundary of $S$ that map 
  to the same segment in $X$ and $D$ contains a cylinder core curve 
  has length at most $d$.  
The boundary of this cylinder forms a separating system of length 
  at most $2d$, so that $d\ge\delta$.  
In any case, there is an embedded square with side $\delta$ at 
  the center of the disk $D$ and we refer to this square as the 
  \emph{central square} associated to $\Delta$.  

For each pair $(\Delta,\gamma)$ where $\Delta$ is a medium or 
  large triangle and $\gamma$ is a long edge on its boundary, 
  we associate an embedded square $S'$ of side $\delta$ that 
  contains the midpoint of $\gamma$ as follows.  
The same argument as before ensures that $S'$ exists.  
Note that $S'$ is $K$-reachable from the central square associated 
  to $\Delta$ for any $K>0$.  
Note also that if $S''$ is the square associated to $(\Delta',\gamma)$ 
  where $\Delta'\in\cD$ is the other triangle having $\gamma$ on its 
  boundary, then the union of the circumscribing disks contains a 
  rectangle that contains $S'\cup S''$, so that $S''$ is $K$-reachable 
  from $S'$ for any $K>0$.  

Let $\cN$ the collection of the central squares associated to any 
  medium or large triangles $\Delta$ together with all the squares 
  associated with all possible pairs $(\Delta,\gamma)$ where $\gamma$ 
  is a long edge on the boundary of a medium or large traingle $\Delta$.  
The number of elements in $\cN$ is bounded above by some $N=N(S)$.  

\begin{lemma}\label{large}
There is a universal constant $c>0$ such that the area of any 
  large triangle is at least $c\delta^4$.  
\end{lemma}
\begin{proof}
Let $\Delta$ be a large triangle, $a$ the length of its shortest 
  side and $\alpha$ the angle opposite $a$.  
The circumsribing disk $D$ has diameter given by $d=a\csc\alpha$.  
If $d$ is large enough, then $D$ contains a maximal cylinder $C$ 
  whose height $h$ and waist $w$ are related by (\cite{MS}) 
  $$ h \le d \le \sqrt{h^2+w^2} < 2h.$$  
Since the diameter of each component of $\Delta\setminus C$ is 
  at most $w$, there is a curve of length at most $3w$ joining 
  two vertices of $\Delta$.  
Hence, $a\le 3w\le \frac{3}{h} \le frac{6}{d}$.  
Since each side of $\Delta$ is at least $\delta$ we have 
  $$\area(\Delta) \ge \frac{1}{2}\delta^2\sin\alpha 
      = \frac{a\delta^2}{2d} \ge \frac{\delta^4}{12}.$$  
\end{proof}

Let $Y$ be the union of all small triangles and short edges in $\cD$.  
Its topological boundary $\partial Y$ is the union of all short edges.  
Let $Y'$ be a component in the complement of $Y$.  
If $Y'$ contains a large triangle, then Lemma~\ref{large} implies it is 
  homotopically nontrivial as soon as $|\partial Y'|^2<4\pi\area(Y')$, 
  which holds if $\eps$ is small enough.  
Otherwise, $Y'$ is a union of medium triangles and is necessarily 
  homeomorphic to an annulus.  The core of this annulus can neither 
  bound a disk nor be homotopic to a punture.  
Therefore, each component in the complement of $Y$ is homotopically 
  nontrivial.  
Assuming $\eps$ is small enough so that $|\partial Y|<\delta$, 
  we conclude that the complement of $Y$ is connected, from which 
  it follows that $\Gamma_K(\cN)$ is connected for any $K>0$.  

If $Y$ has empty interior, then $X$ is $K$-fully covered by $\cN$ 
  for any $K>0$  and we are done.  
Hence, assume $Y$ has nonempty interior $Y^o$ and note that $Y^o$ is 
  homotopically trivial, for otherwise $\partial Y$ would separate.  
To show that $X$ is $K$-fully covered by $cN$ it is enough to show 
  that for any $x\in Y^o$ we can find a vertical segment starting 
  at $x$, of length at most $K\eps$, and having a subsegment of 
  length at least $\eps$ contained in some medium or large triangle.  
This will be achieved by the next three lemmas.  

\begin{lemma}\label{ver1}
The length of any vertical segment contained in $Y^o$ is at most 
  $4M\eps$ where $M$ is the number of small triangles.  
\end{lemma}
\begin{proof}
Suppose not.  Then there exists a vertical segment $\gamma$ of length 
$4M\eps$ contained in $Y^o$.  Since the length of any component of 
$\gamma$ that lies in any small triangle is less than $2\eps$, there 
exists a small triangle $\Delta$ that intersects $\gamma$ in at least 
three subsegments $\gamma_i,i=1,2,3$.  Let $p_i,i=1,2,3$ be the midpoints 
of these segments and assume the indices are chosen so that $p_2$ lies 
on the arc along $\gamma$ joining $p_1$ to $p_3$.  If $\gamma_1$ and 
$\gamma_2$ traverse $\Delta$ in the same direction, we can form an 
essential simple closed curve in $Y^o$ by taking the arc along 
$\gamma$ from $p_1$ to $p_2$ and concatenating it with the arc in 
$\Delta$ from $p_2$ back to $p_1$.  Since $Y^o$ is homotopically trivial, 
we conclude that $\gamma_1$ and $\gamma_2$ traverse $\Delta$ in opposite 
directions.  Similarly, $\gamma_2$ and $\gamma_3$ traverse $\Delta$ in 
opposite directions, so that $\gamma_1$ and $\gamma_3$ traverse $\Delta$ 
in the same direction.  Let $\tau$ be the arc in $\Delta$ that joins the 
midpoints of $\gamma_1$ and $\gamma_3$.  Note that $\tau$ cannot be 
disjoint from $\gamma_2$ for otherwise we can form a essential simple 
closed curve by taking the union of $\tau$ with the arc along $\gamma$ 
joining $p_1$ and $p_3$.  Let $p_2'$ be the point where $\tau$ intersects 
$\gamma_2$ and note that we can form an essential simple closed curve by 
following arc along $\gamma$ from $p_1$ to $p_2'$, followed by the arc in 
$\Delta$ from $p_2'$ to $p_3$, followed by the arc along $\gamma$ from 
$p_3$ back to $p_2'$, then back to $p_1$ along the arc in $\Delta$.  
In any case, we obtain a contradiction to the fact that $Y^o$ is 
homotopically trivial and this contradiction proves the lemma.
\end{proof}

\begin{lemma}\label{ver2}
Suppose $\gamma$ is a vertical segment in the complement of $Y$ which 
  does not pass through any singularity, has length less than $\eps$, 
  and has each of its endpoints in the interior of some short edge in 
  $\partial Y$.  
Then there is a finite collection of triangles such that $\gamma$ is 
  contained in the interior of their union and each triangle is formed 
  by three saddle connections of lengths less than $7\eps$.  
\end{lemma}
\begin{proof}
Let $\tau$ and $\tau'$ be the short edges in $\partial Y$ that contain 
  the endpoints of $\gamma$, respectively.  
Let $\alpha$ be a curve joining one endpoint of $\tau$ to the endpoint 
  of $\tau'$ on the same side of $\gamma$ by following an arc along $\tau$, 
  then $\gamma$, and then another arc along $\tau'$.  
Let $\beta$ be the curve formed using the remaining arc of $\tau$ followed 
  by $\gamma$ and then the remaining arc of $\tau'$.  
Let $\alpha'$ (resp. $\beta'$) be the geodesic representatives in the 
  homotopy class of $\alpha$ (resp. $\beta$) relative to its endpoints.  
Both $\alpha'$ and $\beta'$ is a finite union of saddle connections 
  whose total length is less than $3\eps$.  
The union $\tau\cup\alpha'\cup\tau'\cup\beta'$ bounds a closed set $C$ 
  whose interior can be triangulated using saddle connections, each of 
  which joins a singularity on $\alpha'$ to a singularity on $\beta'$.  
The length of each such interior saddle connection is less than $7\eps$.  
The union of $C$ with the small triangles having $\tau$ and $\tau'$ on 
  their boundary contains $\gamma$ in its interior.  
\end{proof}

\begin{lemma}\label{ver3}
There is a $K'=K'(S)$ such that any vertical segment of length $K'\eps$ 
  intersects the complement of $Y$ in a subsegment of length $\eps$.  
\end{lemma}
\begin{proof}
Let $\Sigma$ be the set of singularities of $(X,q)$.  
By Lemma~\ref{ver1}, there is some $M'=M'(S)$ such that the length of 
  any vertical segment contained in $Y\cup\Sigma$ is less than $K'\eps$.  
Let $K'=(M'+1)\nu^2$ where $\nu=\nu(S)$ is the total number of edges.  
Suppose there exists a vertical segment $\gamma$ of length $K'\eps$ such 
  that each component in the complement of $Y$ has length less than $\eps$.  
Then there are two subsegments of $\gamma$ in the complement of $Y$ that 
  join the same pair of short edges in $\partial Y$.  
Let $Z$ be the complex generated by saddle connections of length less 
  than $7\eps$.  (See \cite{EM}.)  
Its area is $O(\eps^2)$ and its boundary is $O(\eps)$ where the implicit 
  constants depending only on $S$.  
By Lemma~\ref{ver2} implies the interior of $Z$ is homotopically nontrivial, 
  implying that $\partial Z$ forms a separating system.  
If $\eps$ is sufficiently small, this contradicts $\ell_3(X,q)>2\delta$.  
This is a contradiction proves the lemma.  
\end{proof}

Let $M=M(S)$ be the total number of triangles.  
Given $x\in Y^o$ we may form a vertical segment $\gamma$ of length $K'\eps$ 
  with one endpoint at $x$.  
By Lemma~\ref{ver3}, there is a component of $\gamma$ is the complement of 
  $Y$ whose length is at least $\eps$.  
This component is a union of at most $M$ segments, each of which contained 
  in some medium or large triangle.  
The longest such segment has length at least $\frac{\eps}{M}$.  
Hence, $x$ is $K$-reachable from the central square associated to the 
  medium or large triangle that contains this segment, where $K=K'M$.  
This complete the proof of Proposition~\ref{prop:network}.  
\end{proof}

Boshernitzan's criterion \cite{Ve2} is a consequence of Masur's theorem 
  \cite{Ma1} by the first inequality.  Masur's theorem is a consequence 
  of the following by the second inequality.  
\begin{theorem}\label{thm:recurrent}
If $\limsup_{t\to\infty}\ell_3(X_t)>0$ then $\cF_v$ is uniquely ergodic.  
\end{theorem}
\begin{proof}
Fix $t_n\to\infty$ and $\delta_3>0$ such that $\ell_3(X_n)>2\delta_3$ 
  for all $n$.  
Let $\cD_n$ be a complete Delaunay triangulation of $X_n$ and 
  let $\lambda_i^n$ be the length of the $i$th shortest edge.  
By convention, we set $\lambda_0^n=0$ for all $n$.  
Let $i\ge0$ be the unique index determined by 
\begin{equation}
  \liminf_n \lambda_i^n=0,\quad\text{and}\quad 
  \liminf_n \lambda_{i+1}^n>0.  
\end{equation}
By passing to a subsequence and re-indexing, 
  we may assume there is a $\delta_1>0$ such that 
\begin{equation}
  \lambda_{i+1}^n>\delta_1 \quad\text{for all $n$.}
\end{equation}
Assume $n$ is large enough so that 
\begin{equation}
  \lambda_i^n < \eps  
\end{equation}
  where $\eps$ is small enough as required by Proposition~\ref{prop:network} 
  with $\delta=\min(\delta_1,\delta_3)$.  
The theorem now follows from Propositions~\ref{prop:Masur}.  
\end{proof}

\section{Rates of Divergence}\label{S:Rates}
In this section we discuss the various notions of divergence 
and the rates of divergence.  

The hypothesis of Theorem~\ref{thm:main} can be formulated in 
  terms of the flat metric on $X$ without appealing to the 
  forward evolution of the surface.  
Let $h(\gamma)$ and $v(\gamma)$ denote the horizontal and vertical 
  components of a saddle connection $\gamma$, which are defined by 
$$h(\alpha)=\left|\Re\int_\gamma\omega\right|\quad\text{and}\quad
  v(\alpha)=\left|\Im\int_\gamma\omega\right|.$$
It is not hard to show that the following statements are equivalent.  
\begin{enumerate}
  \item[(a)] There is a $C>0$ such that for all $t>0$, $d(t)<\eps\log t+C.$ 
  \item[(b)] There is a $c>0$ such that for all $t>0$, $\ell(X_t)>c/t^{\eps/2}.$ 
  \item[(c)] There are constants $c'>0$ and $h_0>0$ such that for all 
        saddle connections $\gamma$ satisfying $h(\gamma)<h_0$, 
  \begin{equation}\label{strong:diophantine}
        h(\gamma)>\frac{c'}{v(\gamma)(\log v(\gamma))^\eps}.
  \end{equation}
\end{enumerate}


For any $p>1/2$ there are translation surfaces with nonergodic $\cF_v$ 
  whose Teichm\"uller geodesic $X_t$ satisfies the sublinear slow rate 
  of divergence $d(t)\le Ct^p$.  See \cite{Ch2}.  
Our main result asserts a logarithmic slow rate of divergence is enough 
  to ensure unique ergodicity of $\cF_v$.

\section{Slow divergence}\label{S:Slow}
Our interest lies in the case where $\ell_1(X_t)\to0$ as $t\to\infty$.  
To prove of Theorem~\ref{thm:main} we shall need an analog of 
  Proposition~\ref{prop:Masur} that applies to a continuous family 
  of networks $\cN_t$ whose squares have dimensions going to zero.  
We also need to show that the slow rate of divergence gives us some 
  control on the rate at which the small squares approach zero.  

A crucial assumption in the proof of Theorem~\ref{thm:recurrent} is 
  that the squares in the networks have area bounded away from zero.  
This allowed us to find generic points that persist in the squares 
  of the networks along a subsequence $t_n\to\infty$.  
If the area of the squares tend to zero slowly enough as $t\to\infty$, 
  one can still expect to find persistent generic points, with the 
  help of the following result from probability theory.  

\begin{lemma}\label{lem:PZ}
\textbf{(Paley-Zygmund \cite{PZ})}
If $A_n$ be a sequence of measureable subsets of a probability 
  space satisfying 
  \begin{enumerate}
    \item[(i)] $|A_n\cap A_m|\le K|A_n||A_m|$ for all $m>n$, and  
    \item[(ii)] $\sum |A_n| = \infty$ 
  \end{enumerate}
  then $|A_n \text{~i.o.}|\ge1/K$.  
\end{lemma}

\begin{definition}\label{def:buffer}
We say a rectangle is $\alpha$-\emph{buffered} if it can be extended 
  in the vertical direction to a larger rectangle of area at least 
  $\alpha$ that overlaps itself at most once.  
(By the area of the rectangle, we mean the product of its sides.)  
\end{definition}

\begin{proposition}\label{prop:subseq}
Suppose that for every $t>0$ we have an $\alpha$-buffered square $S_t$ 
  embedded in $X_t$ with side $\sigma_t>c/t^{\eps/2}$, $0<\eps\le1$.  
Then there exists $t_n\to\infty$ and $K=K(c,\alpha)$ such that 
  $A_n=f^{-1}S_n$ satisfies $|A_n\cap A_m|\le K|A_n||A_m|$ for all $m>n$ 
  and $t_n\in O(n\log n\log\log n)$.  
\end{proposition}
\begin{proof}
Let $(t_n)$ be any sequence satisfying the recurrence relation 
  $$t_{n+1}=t_n+\eps\log(t_{n+1}) \qquad t_0>1.$$  
Note that the function $y=y(x)=x-\eps\log x$ is increasing for $x>\eps$ 
  and has inverse $x=x(y)$ is increasing for $y>1$, from which it follows 
  that $(t_n)$ is increasing.  We have  $$\sigma_m e^{t_m-t_n} 
   >(ct_m^{-\eps/2})t_{n+1}^{\eps}\cdots t_m^{\eps/2}>ct_n^{\eps/2}.$$  
Let $B_n\supset S_n$ be a rectangle in $X_n$ that has the same 
  width as $S_n$ and area at least $\alpha$.  
Since $B_n$ overlaps itself at most once, $\alpha\le2|B_n|\le2$.  
Therefore, the height of $B_n$ is $<2/\sigma_n<(2/c)t_n^{\eps/2}$, 
  which is less than $2/c^2$ times the height of the rectangle 
  $R_m=f_n\circ f_m^{-1}S_m$, by the choice of $t_n$.  
Let $R'_m$ be the smallest rectangle containing $R_m$ that has 
  horizontal edges disjoint from the interior of $B_n$.  
Its height is at most $1+4/c^2$ times that of $R_m$.  
For each component $I$ of $S_n\cap R'_m$ there is a corresponding 
  component $J$ of $B_n\cap R'_m$ (see Figure~\ref{fig:buffer}) 
  so that 
$$|A_n\cap A_m|=\sum|I|\le\frac{\alpha(S_n)}{\alpha}\sum|J|
  \le\alpha^{-1}|S_n|\alpha(R'_m)<\frac{4+c^2}{\alpha c^2}|A_n||A_m|.$$  
Choose $t_0$ large enough so that $t_{n+1}<2t_n$ for all $n$ and suppose 
  that for some $C>0$ and $n>1$ we have $t_n<Cn\log n\log\log n$.  Then 
\begin{align*}
  t_{n+1} &< t_n+\log t_{n+1} < t_n + \log t_n + \log 2 \\
     &< Cn\log n \log\log n + \log n + \log\log n + \log\log\log n + \log 2C \\
     &< Cn\log n \log\log n + \log n \log\log n \qquad\text{for $n\gg1$} \\
     &< C(n+1)\log(n+1)\log\log(n+1) 
\end{align*}
  so that $t_n\in O(n\log n\log\log n)$.  
\end{proof}

\begin{remark}
The condition $\liminf t^{1/2}\ell_1(X_t)>0$ (corresponding to $\eps=1$ above) 
  holds for almost every direction in \emph{every} Teichm\"uller disk.  
\cite{Ma2}
\end{remark}

\begin{figure}
\begin{center}
\includegraphics{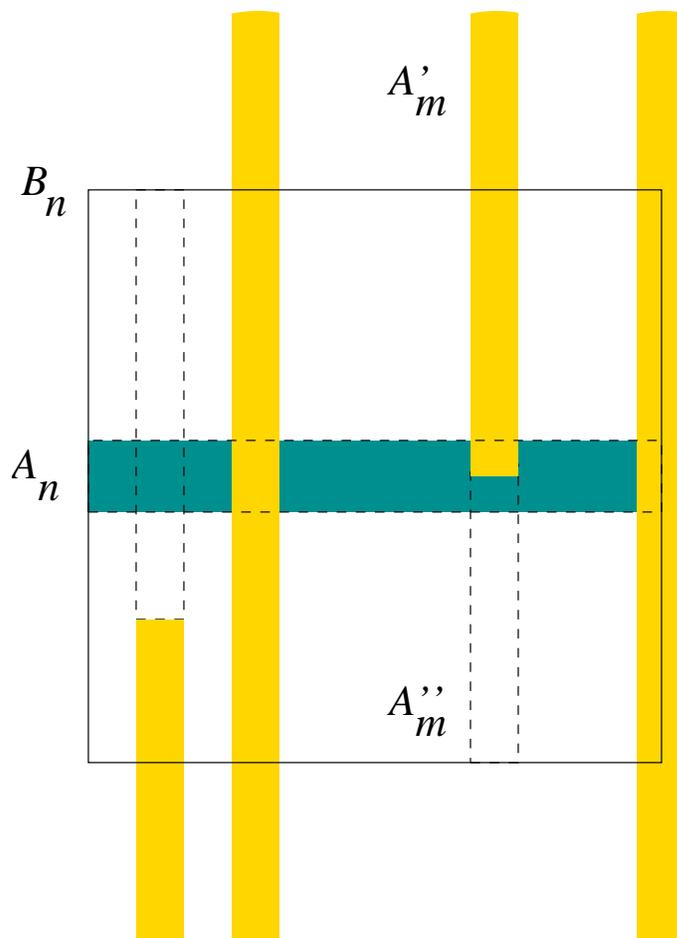}
\caption{For purposes of illustration, the rectangles are represented 
  by their images in $X_t$ where $t$ is the unique time when $B_n$ 
  maps to a square under the composition $f_m\circ f_n^{-1}$ of 
  Teichm\"uller maps.}\label{fig:buffer}
\end{center}
\end{figure}

\begin{definition}\label{def:strip}
Assume the vertical foliation of $(X,q)$ is minimal.  
Given a saddle connection $\gamma$, we may extend each 
critical leaf until the first time it meets $\gamma$.  
Let $\Gamma$ be the union of these critical segments 
with $\gamma$.  By a \emph{vertical strip} we mean any 
component in the complement of $\Gamma$.  We refer to 
any segment along a vertical edge on the boundary of 
a vertical strip that joins a singularity to a point 
in the interior of $\gamma$ as a \emph{zipper}.  
\end{definition}

Each vertical strip has a pair of edges contained in $\gamma$ 
as well as a pair of vertical edges, each containing exactly 
one singularity.  
The number $m$ of vertical strips determined by a saddle 
  connection depends only on the stratum of $(X,q)$.  
Thus, any rectangle containing the vertical strip with 
  most area has area is a $1/m$-buffer for any square of 
  the same width contained in it.  
See Figure~\ref{fig:strip}.  

\begin{figure}
\begin{center}
\includegraphics{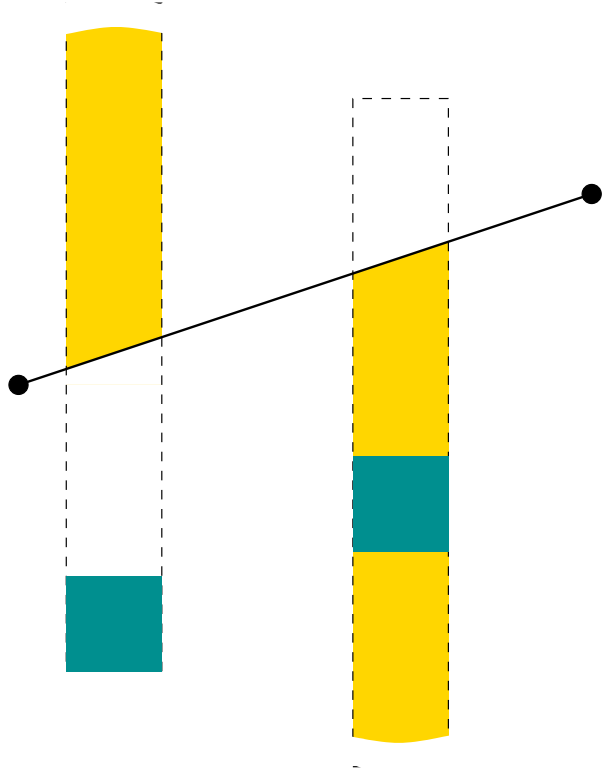}
\caption{Any rectangle containing the vertical strip with most area 
  serves as a buffer for any square of the same width contained in it.}
\label{fig:strip}
\end{center}
\end{figure}

The condition (\ref{strong:diophantine}) prevents the slopes 
  of saddle connections from being too close to vertical.  
This allows for some control on the widths of vertical strips.  
\begin{proposition}\label{prop:strip}
Let $h_0>0$, $c>0$ and $\eps>0$ be the constants of the 
  Diophantine condition satisfied by $(X,q)$.  
Let $m$ be the number of vertical strips determined by 
  any saddle connection.  
For any $\kappa>1$ and $\delta>2m\eps$ there exists a 
  $t_0=t_0(h_0,c,\eps,\kappa,\delta)>1$ such that 
  for any $t>t_0$ and any saddle connection $\gamma$ in $X_t$ 
  whose length is at most $\kappa$, the width of any vertical 
  strip determined by $\gamma$ is at least $t^{-\delta}$.  
\end{proposition}
\begin{proof}
Without loss of generality we may assume $c<1$.  
The value of $t_0$ is chosen large enough to satisfying various 
  conditions that will appear in the course of the proof.  
In particular, we require $t_0$ be large enough so that for 
  $j=1,\dots,2m$ and any $t>t_0$ we have 
\begin{equation}\label{ieq:cj}
  c^{2j-1}t^{\delta-j\eps}-\kappa>c^{2j}t^{\delta-j\eps}.  
\end{equation}
First, observe that for any saddle connection $\gamma'$ in $X_t$ 
\begin{equation}\label{ieq:dio}
  h(\gamma')\le h(\gamma)\quad\text{and}\quad v(\gamma')\le 
    e^{t/2}\quad\Rightarrow\quad h(\gamma')v(\gamma')>ct^{-\eps}.
\end{equation}
Indeed, if $\kappa e^{-t_0/2}<h_0$ we can apply the Diophantine 
  condition to the saddle connection $\gamma'_0$ in $X$ that 
  corresponds to $\gamma'$ to conclude 
$$h(\gamma')v(\gamma')=h(\gamma'_0)v(\gamma'_0)
     >\frac{c}{(\log v(\gamma'_0))^\eps}
     =\frac{c}{(t/2+\log v(\gamma'))^\eps}\ge ct^{-\eps}.$$

We shall argue by contradiction and suppose that there is 
  a vertical strip $P_1$ supported on $\gamma$ whose width 
  is $<t^{-\delta}$.  
Let $\gamma_1$ be the saddle connection joining the 
  singularities on its vertical edges.  
If $v(\gamma_1)\le e^{t/2}$, then (\ref{ieq:dio}) implies 
  $v(\gamma_1)>ct^{\delta-\eps}$.  If $v(\gamma_1)>e^{t/2}$ 
  we get the same conclusion by choosing $t_0$ large enough.  
We shall consider only zippers that protrude from a fixed 
  side of $\gamma$.  
Using (\ref{ieq:cj}) with $j=1$ we see that the height $a_1$ 
  of the longer zipper on the boundary of $P_1$ satisfies 
  $$a_1>c^{\delta-\eps}-\kappa > c^2t^{\delta-\eps}.$$
Suppose we have a contraption $P_1\cup\dots\cup P_j,j\ge1$ 
  of vertical strips joined along zippers of height 
  $a_1,\dots,a_{j-1}$ and such that on the boundary of the 
  contraption there is a zipper of height $a_j$ satisfying 
\begin{equation}\label{ieq:height}
  \min(a_1,\dots,a_j)>c^{2j}t^{\delta-j\eps}.  
\end{equation}
If the total width $w_j$ of the contraption is less than 
  that of $\gamma$, we can adjoin a vertical strip $P_{j+1}$ 
  along the zipper of height $a_j$.  
The new contraption $P_1\cup\dots\cup P_{j+1}$ contains an 
  embedded parallelogram with a pair of vertical sides of 
  length greater than the RHS of (\ref{ieq:height}) and 
  whose width equals the total width $w_{j+1}$ of the new 
  contraption.  
Therefore, 
\begin{equation}\label{ieq:width}
   w_{j+1}<c^{-2j}t^{-(\delta-j\eps)} 
\end{equation}
Let $a_{j+1}$ be the height of the longer zipper on the 
  boundary of the new contraption.  

\textbf{Claim:} $a_{j+1}>c^{2(j+1)}t^{\delta-(j+1)\eps}.$

If not, we can find a saddle connection $\gamma_{j+1}$ that 
  crosses from one vertical boundary of the union to the 
  other, with vertical component satisfying 
  $v(\gamma_{j+1})<c^{2(j+1)}t^{\delta-(j+1)\eps}+\kappa<
    c^{2j+1}t^{\delta-(j+1)\eps}$ by virtue of (\ref{ieq:cj}), 
    assuming $j+1\le m$.  
But then (\ref{ieq:dio}) implies $$w_{j+1}>
    \frac{c}{v(\gamma_{j+1}t^\eps}>c^{-2j}t^{-(\delta-j\eps)}$$
   which contradicts (\ref{ieq:width}), and thus establishes 
   the claim.  

As soon as the total width of the new contraption equals that 
  of $\gamma$, both zippers on the boundary are degenerate or 
  have height zero, contradicting the claim.  
This contradiction implies width of $P_1$ is at least $t^{-\delta}$.  
\end{proof}

\begin{theorem}\label{thm:slow:div}
There exists $\eps>0$ depending only on the stratum of $(X,q)$ 
  such that $\liminf t^\eps\ell_1(X_t)>0$ implies $\cF_V$ is 
  uniquely ergodic.  
\end{theorem}
\begin{proof}
Let $\cD_t$ be a complete Delaunay triangulation of $X_t$.  
If a triangle $\Delta\in\cD_t$ has an edge $\gamma$ of length 
  $\ell>\sqrt{2/\pi}$ then the circumscribing disk contains a 
  maximal cylinder $C$ that is crossed by $\gamma$ and such 
  that $h\le \ell\le \sqrt{h^2+c^2}$ where $h$ and $c$ are 
  height and circumference of the cylinder $C$ (\cite{MS}).  
Since $hc=\area(C)\le1$ a long Delaunay edge will cross a 
  cylinder of large modulus.  
Let $\kappa$ and $\mu$ be chosen so that a Delaunay edge 
  of length $>\kappa$ crosses a cylinder of modulus $>\mu$ 
  and assume $\mu$ is large enough so that the cylinder 
  crossed by the Delaunay edge is uniquely determined.  

For each Delaunay edge $\gamma$ of length at most $\kappa$, 
  let $\Delta$ and $\Delta'$ be the Delaunay triangles that 
  have $\gamma$ on its boundary, and let $D$ and $D'$ the 
  respective circumscribing disks.  
Applying Proposition~\ref{prop:strip} we can find a vertical 
  strip of area at least $\delta$ ($=\frac{1}{m}$) which is 
  contained in some immersed rectangle than contains a square 
  $S$ of side $\sigma_t=ct^{-\eps}$ centered at some point on 
  the equator of $D$ as well as a square $S'$ of the same size 
  centered at some point on the equator of $D'$.  
Both $S$ and $S'$ are $\delta$-buffered and $K$-reachable from 
  each other for any $K>0$.  

Call an edge in $\cD_t$ \emph{long} if it crosses a cylinder 
  of modulus $>\mu$.  
Call a triangle in $\cD_t$ \emph{thin} if it has two long edges.  
We note that if a triangle in $\cD_t$ has any long edges on its 
  boundary, then it has exactly two such edges.  
Each cylinder $C$ of modulus $>\mu$ determines a collection 
  of long edges and thin triangles whose union contains $C$.  
Moreover, the intersection of the disks circumscribing the 
  associated thin triangles contains a $\kappa/2$-neighborhood 
  of the core curve of $C$.  
For each such $C$, choose a saddle connection on its boundary 
  and apply Proposition~\ref{prop:strip} to construct a 
  $\delta$-buffered square $S''$ of side $\sigma_t$ centered 
  at some point on the core curve of $C$.  

Let $\cN_t$ be the collection of all squares $S$ and $S'$ 
  associated with edges in $\cD_t$ of length at most $\kappa$ 
  together with all the squares $S''$ associated cylinders 
  of modulus $>\mu$.  
It is easy to see that $\cN_t$ is a $K$-network for any $K>0$ 
  and the number of elements in $\cN_t$ is bounded above by 
  some constant $N$ that depends only on the stratum.  
Choose $\eps$ so that $\eps N\le1$ and let $t_n\to\infty$ be 
  the sequence given by Proposition~\ref{prop:subseq}, and 
  note that $\sum_n\sigma_n^2=\infty$ by the choice of $\eps$.  
Let $S_i^n,i=1,\dots,N$ enumerate the elements of $\cN_n$, 
  using repetition, if necessary.  
Aapplying Lemma~\ref{lem:PZ} to the subsets 
  $A_n=f_n^{-1}S_1^n\times\dots\times f_n^{-1}S_N^n$ 
  of the probability space $X^N$, we obtain an $N$-tuple of 
  generic points $(x_1,\dots,x_N)$ with the property that 
  for infinitely many $n$, $f_nx_i\in S_i^n$ for all $i$.  
By passing to a subsequence, we may assume this holds for 
  every $n$.  

Let $\nu$ be the ergodic component that contains $x_1$.  
Since $\Gamma_K(\cN_n)$ is connected, we can find for each $n$ an 
  $x_i,i\neq1$ such that $f_nx_i$ is $K$-reachable from $f_nx_1$.  
After re-indexing, if necessary, we may $f_nx_2$ is $K$-reachable 
  from $f_nx_1$ for infinitely many $n$, and by further passing 
  to a subsequence, we may assume this holds for every $n$.  
By Lemma~\ref{lem:visible} it follows that $x_2$ belongs to 
  the ergodic component $\nu$.  
Given $x_1,\dots,x_i,i<N$ we can find for each $n$ an $x_j,j>i$ 
  such that $f_nx_j$ is $K$-reachable from $f_n\{x_1,\dots,x_i\}$.  
After re-indexing and passing to a subsequence, we may assume 
  that $f_nx_{i+1}$ is $K$-reachable from $f_n\{x_1,\dots,x_i\}$.  
Proceeding inductively, we deduce that each $x_i$ belongs to 
  the ergodic component $\nu$.  

Since $X_n$ is $K$-fully covered by $\cN_n$, given any generic 
  point $z$, we can find for each $n$ an $x_i$ such that $f_nz$ 
  is $K$-visible from $f_nx_i$.  
For some $i$ this holds for infinitely many $n$, so that by 
  Lemma~\ref{lem:visible}, $z$ belongs to the same ergodic 
  component as $x_i$, i.e. $\nu$.  
This shows that $\cF_v$ is uniquely ergodic.  
\end{proof}

\begin{remark}
Given any function $r(t)\to\infty$ as $t\to\infty$ there exists 
  a Teichm\"uller geodesic $X_t$ determined by a nonerogdic vertical 
  foliation such that $\limsup_{t\to\infty}r(t)\ell_1(X_t)>0$.  
See \cite{CE}.  
\end{remark}

\section{Nonergodic directions}\label{S:Veech}
In the case of double covers of the torus branched over two 
  points, it is possible to give a complete characterisation 
  of the set of nonergodic directions.  
This allows us to obtain an affirmative answer to a question 
  of W.~Veech (\cite{Ve1}, p.32, question 2).  
We briefly sketch the main ideas of this argument.  

Let $(X,q)$ be the double of the flat torus $T=(\C^2/\Z[i],dz^2)$ 
  along a horizontal slit of length $\lambda, 0<\lambda<1$.  
Let $z_0,z_1\in T$ be the endpoints of the slit.  
The surface $(X,q)$ is a branched double cover of $T$, branched 
  over the points $z_0$ and $z_1$.  
Assume $\lambda\not\in\Q$.  
(If $\lambda\in\Q$, the surface is square-tiled, hence Veech; 
  in this case, a direction is uniquely ergodic iff its slope 
  is irrational.)  
Let $V$ be the set of holonomy vectors of saddle connections in $X$.  
Then $$V=W\cup Z$$ where $Z$ is the set of holonomy vectors of 
  simple closed curves in $T$ and $W$ is the set of holonomy 
  vectors of the form $\lambda+m+ni$ where $m,n\in\Z$.  
We refer to saddle connections with holonomy in $W$ as \emph{slits} 
  and those with holonomy in $Z$ as \emph{loops}.  
Given any slit $w\in W$, there is segment in $T$ joining $z_0$ 
  to $z_1$ whose holonomy vector is $w$.  
The double of $T$ along this segment is a branched cover that 
  is biholomorphically equivalent to $(X,q)$ if and only if 
  $$w\in W_0:=\{\lambda+m+ni: m,n\in2\Z\}.$$ 
(See \cite{Ch1}.)  In this case, the segment lifts to a pair 
  of slits that are interchanged by the covering transformation 
  $\tau:X\to X$ and the complement of their union is a pair of 
  slit tori also interchanged by the involution $\tau$.  
A slit is called \emph{separating} if its holonomy lies in $W_0$; 
  otherwise, it is \emph{non-separating}.  

Fix a direction $\theta$ and let $\cF_\theta$ be the foliation in 
  direction $\theta$.  
Let $X_t$ be the Teichm\"uller geodesic determined by $\cF_\theta$.  
Let $\ell(X_t)$ denote the length of the shortest saddle connection 
  measured with respect to the sup norm.  
Assume $\lim_{t\to\infty}\ell(X_t)=0$ for otherwise the length of the 
  shortest separating system is bounded away from zero along some 
  sequence $t_n\to\infty$ and Theorem~\ref{thm:recurrent} implies 
  $\cF_\theta$ is uniquely ergodic.  
Note that the shortest saddle connection can always be realised by 
  either a slit or a loop, and if $t$ is sufficiently large, we may 
  also choose it to have holonomy vector with positive imaginary part.  
Note also that $\log\ell(X_t)$ is a piecewise linear function of 
  slopes $\pm1$.  

Let $v_j$ be the sequence of slits or loops that realise the local 
  minima of $-\log\ell(X_t)$ as $t\to\infty$.  
We assume $\theta$ is minimal so that this is an infinite sequence.  
If $v_j$ is a loop, then $v_{j+1}$ must be a slit since two loops 
  cannot be simultaneously short.  
If $v_j$ and $v_{j+1}$ are both slits, then the length of the vector 
  $v_{j+1}-v_j$ at the time when $v_j$ and $v_{j+1}$ have the same 
  length (with respect to the sup norm) is less than twice the common 
  length.  
It follows that $v=v_{j+1}-v_j\in Z$ so that either $v_j$ or $v_{j+1}$ 
  is non-separating.  
Note that $v$ is the shortest shortest loop at this time.  
If $v_j$ (resp. $v_{j+1}$) is non-separating, then at the first (resp. 
  last) time when $v$ is the shortest loop, there is another loop $v'$ 
  that forms an integral basis for $\Z[i]$ together with $v$.  
Since the common length of these loops is at least one, $v_j$ (resp. 
  $v_{j+1}$) is the unique slit or loop of length less than one and 
  it follows that the length of the shortest separating system is 
  at least one.  
If there exists an infinite sequence of pairs of consecutive slits, 
  then we may apply Theorem~\ref{thm:recurrent} to conclude that 
  $\cF_\theta$ is uniquely ergodic.  

It remains to consider the case when the sequence of shortest vectors 
  alternates between separating slits and loops 
  $$ \dots v_{j-1}, w_j, v_j, w_{j+1}, v_{j+1}, \dots $$ 
  with increasing imaginary parts.  
Note that $w_{j+1}-w_j$ is an even positive multiple of $v_j$, say $2b_{j+1}$.  
The surface $X_t$ at the time $t_j$ when $v_j$ is shortest (slope $\pm1$) 
  can be described quite explicitly.  
The slit $w_j$ is almost horizontal while $w_{j+1}$ is almost vertical.  
The area exchange between the partitions determined by $w_j$ and $w_{j+1}$ 
  is approximately $|v_j\times w_j|=|v_j\times w_{j+1}|=\delta_j$.  
The surface can be represented as a sum of tori slit along $w_{j+1}$, 
  each containing a cylinder having $v_j$ as its core curve and 
  occupying most of the area of the slit torus.  
Using this representation, we can find a single buffered square $S_j$ 
  with side $\sqrt{\delta_j}$ which, together with its image under 
  $\tau$, forms a $K$-network (for any $K>0$).  
A straightforward calculation shows that the sequence $A_j=f_j^{-1}S_j$ 
  satisfies $|A_j\cap A_k|\le K|A_j||A_k|$ for all large enough $j<k$.  
Lemma~\ref{lem:PZ} now implies $\cF_\theta$ is uniquely ergodic if 
  $\sum |A_j|=\infty$.  

Conversely, if $\sum|A_j|<\infty$ then another straightforward 
  calculation shows that the hypotheses of the nonergodicity 
  criterion in \cite{MS} are satisfied, implying that $\cF_\theta$ 
  is nonergodic.

\end{document}